\documentclass[oneside,english]{amsart}
\usepackage[T1]{fontenc}
\usepackage[latin9]{inputenc}
\usepackage{babel}
\usepackage{verbatim}
\usepackage{amstext}
\usepackage{amsthm}
\usepackage{amssymb}
\usepackage[unicode=true,pdfusetitle,
 bookmarks=true,bookmarksnumbered=false,bookmarksopen=false,
 breaklinks=false,pdfborder={0 0 1},backref=section,colorlinks=false]
 {hyperref}

\makeatletter

\providecommand{\tabularnewline}{\\}

\numberwithin{equation}{section}
\numberwithin{figure}{section}
\theoremstyle{plain}
\newtheorem{thm}{\protect\theoremname}[section]
  \theoremstyle{remark}
  \newtheorem{rem}[thm]{\protect\remarkname}
  \theoremstyle{definition}
  \newtheorem{defn}[thm]{\protect\definitionname}
  \theoremstyle{plain}
  \newtheorem{fact}[thm]{\protect\factname}
  \theoremstyle{definition}
  \newtheorem{example}[thm]{\protect\examplename}
  \theoremstyle{plain}
  \newtheorem{lem}[thm]{\protect\lemmaname}
  \theoremstyle{plain}
  \newtheorem{cor}[thm]{\protect\corollaryname}
  \theoremstyle{plain}
  \newtheorem{prop}[thm]{\protect\propositionname}
  \theoremstyle{plain}
  \newtheorem{question}[thm]{\protect\questionname}
  \theoremstyle{remark}
  \newtheorem{claim}[thm]{\protect\claimname}
  \theoremstyle{remark}
  \newtheorem*{claim*}{\protect\claimname}

\usepackage{amsfonts}
\usepackage{nicefrac}
\usepackage{amscd}

\usepackage{a4wide}
\usepackage{url}

\AtBeginDocument{
  
}

\makeatother

  \providecommand{\claimname}{Claim}
  \providecommand{\corollaryname}{Corollary}
  \providecommand{\definitionname}{Definition}
  \providecommand{\examplename}{Example}
  \providecommand{\factname}{Fact}
  \providecommand{\lemmaname}{Lemma}
  \providecommand{\propositionname}{Proposition}
  \providecommand{\questionname}{Question}
  \providecommand{\remarkname}{Remark}
\providecommand{\theoremname}{Theorem}

\begin{document}
\global\long\def\code#1{\ulcorner#1\urcorner}
\global\long\def\p{\mathbf{p}}
\global\long\def\q{\mathbf{q}}
\global\long\def\C{\mathfrak{C}}
\global\long\def\SS{\mathcal{P}}
 \global\long\def\pr{\operatorname{pr}}
\global\long\def\image{\operatorname{im}}
\global\long\def\otp{\operatorname{otp}}
\global\long\def\dec{\operatorname{dec}}
\global\long\def\suc{\operatorname{suc}}
\global\long\def\pre{\operatorname{pre}}
\global\long\def\qe{\operatorname{qf}}
 \global\long\def\ind{\operatorname{ind}}
\global\long\def\Nind{\operatorname{Nind}}
\global\long\def\lev{\operatorname{lev}}
\global\long\def\Suc{\operatorname{Suc}}
\global\long\def\HNind{\operatorname{HNind}}
\global\long\def\minb{{\lim}}
\global\long\def\concat{\frown}
\global\long\def\cl{\operatorname{cl}}
\global\long\def\tp{\operatorname{tp}}
\global\long\def\id{\operatorname{id}}
\global\long\def\cons{\left(\star\right)}
\global\long\def\qf{\operatorname{qf}}
\global\long\def\ai{\operatorname{ai}}
\global\long\def\dtp{\operatorname{dtp}}
\global\long\def\acl{\operatorname{acl}}
\global\long\def\nb{\operatorname{nb}}
\global\long\def\limb{{\lim}}
\global\long\def\leftexp#1#2{{\vphantom{#2}}^{#1}{#2}}
\global\long\def\intr{\operatorname{interval}}
\global\long\def\atom{\emph{at}}
\global\long\def\I{\mathfrak{I}}
\global\long\def\uf{\operatorname{uf}}
\global\long\def\ded{\operatorname{ded}}
\global\long\def\Ded{\operatorname{Ded}}
\global\long\def\Df{\operatorname{Df}}
\global\long\def\Th{\operatorname{Th}}
\global\long\def\eq{\operatorname{eq}}
\global\long\def\Aut{\operatorname{Aut}}
\global\long\def\ac{ac}
\global\long\def\DfOne{\operatorname{df}_{\operatorname{iso}}}
\global\long\def\modp#1{\pmod#1}
\global\long\def\sequence#1#2{\left\langle #1\,\middle|\,#2\right\rangle }
\global\long\def\set#1#2{\left\{  #1\,\middle|\,#2\right\}  }
\global\long\def\Diag{\operatorname{Diag}}
\global\long\def\Nn{\mathbb{N}}
\global\long\def\mathrela#1{\mathrel{#1}}
\global\long\def\twiddle{\mathord{\sim}}
\global\long\def\mathordi#1{\mathord{#1}}
\global\long\def\Qq{\mathbb{Q}}
\global\long\def\dense{\operatorname{dense}}
\global\long\def\Rr{\mathbb{R}}
 \global\long\def\cof{\operatorname{cf}}
\global\long\def\tr{\operatorname{tr}}
\global\long\def\treeexp#1#2{#1^{\left\langle #2\right\rangle _{\tr}}}
\global\long\def\x{\times}
\global\long\def\forces{\Vdash}
\global\long\def\Vv{\mathbb{V}}
\global\long\def\Uu{\mathbb{U}}
\global\long\def\tauname{\dot{\tau}}
\global\long\def\ScottPsi{\Psi}
\global\long\def\cont{2^{\aleph_{0}}}
\global\long\def\MA#1{{MA}_{#1}}
\global\long\def\rank#1#2{R_{#1}\left(#2\right)}
\global\long\def\cal#1{\mathcal{#1}}

\def\Ind#1#2{#1\setbox0=\hbox{$#1x$}\kern\wd0\hbox to 0pt{\hss$#1\mid$\hss} \lower.9\ht0\hbox to 0pt{\hss$#1\smile$\hss}\kern\wd0} 
\def\Notind#1#2{#1\setbox0=\hbox{$#1x$}\kern\wd0\hbox to 0pt{\mathchardef \nn="3236\hss$#1\nn$\kern1.4\wd0\hss}\hbox to 0pt{\hss$#1\mid$\hss}\lower.9\ht0 \hbox to 0pt{\hss$#1\smile$\hss}\kern\wd0} 
\def\nind{\mathop{\mathpalette\Notind{}}} 

\global\long\def\ind{\mathop{\mathpalette\Ind{}}}
\global\long\def\Age{\operatorname{Age}}
\global\long\def\lex{\operatorname{lex}}
\global\long\def\len{\operatorname{len}}

\global\long\def\dom{\operatorname{Dom}}
\global\long\def\res{\operatorname{res}}
\global\long\def\alg{\operatorname{alg}}
\global\long\def\dcl{\operatorname{dcl}}
 \global\long\def\nind{\mathop{\mathpalette\Notind{}}}
\global\long\def\average#1#2#3{Av_{#3}\left(#1/#2\right)}
\global\long\def\Ff{\mathfrak{F}}
\global\long\def\mx#1{Mx_{#1}}
\global\long\def\maps{\mathfrak{L}}

\global\long\def\Esat{E_{\mbox{sat}}}
\global\long\def\Ebnf{E_{\mbox{rep}}}
\global\long\def\Ecom{E_{\mbox{com}}}
\global\long\def\BtypesA{S_{\Bb}^{x}\left(A\right)}
\global\long\def\DenseTrees{T_{dt}}

\global\long\def\init{\trianglelefteq}
\global\long\def\fini{\trianglerighteq}
\global\long\def\Bb{\cal B}
\global\long\def\Lim{\operatorname{Lim}}
\global\long\def\Succ{\operatorname{Succ}}

\global\long\def\SquareClass{\cal M}
\global\long\def\leqstar{\leq_{*}}
\global\long\def\average#1#2#3{Av_{#3}\left(#1/#2\right)}
\global\long\def\cut#1{\mathfrak{#1}}
\global\long\def\NTPT{\text{NTP}_{2}}
\global\long\def\Zz{\mathbb{Z}}
\global\long\def\TPT{\text{TP}_{2}}
\global\long\def\supp{\operatorname{supp}}

\global\long\def\OurSequence{\mathcal{I}}
\global\long\def\SUR{SU}
\global\long\def\ShiftGraph#1#2{Sh_{#2}\left(#1\right)}
\global\long\def\ShiftGraphHalf#1{\cal G_{#1}^{\frac{1}{2}}}
\global\long\def\SymShiftGraph#1#2{Sh_{#2}^{sym}\left(#1\right)}
\global\long\def\aa{\textsf{aa}}
\global\long\def\Mm{\mathbb{M}}
\global\long\def\stat{\textsf{stat}}

\title{Local character of Kim-independence}

\author{Itay Kaplan, Nicholas Ramsey, and Saharon Shelah}

\thanks{Partially supported by European Research Council grant 338821. No.
1118 on third author's list.}

\thanks{The first-named author would like to thank the Israel Science Foundation
for partial support of this research (Grant no. 1533/14).}

\begin{abstract}
We show that NSOP$_{1}$ theories are exactly the theories in which
Kim-independence satisfies a form of local character. In particular,
we show that if $T$ is NSOP$_{1}$, $M\models T$, and $p$ is a
type over $M$, then the collection of elementary substructures of
size $\left|T\right|$ over which $p$ does not Kim-fork is a club
of $\left[M\right]^{\left|T\right|}$ and that this characterizes
NSOP$_{1}$. 

We also present a new phenomenon we call dual local-character for
Kim-independence in NSOP$_{1}$-theories. 
\end{abstract}

\maketitle

\section{Introduction}

A well-known theorem of Kim and Pillay characterizes the simple theories
as those theories with an independence relation satisfying certain
properties and shows that, moreover, any such independence relation
must coincide with non-forking independence. As the theory of simple
theories was being developed, work of Chatzidakis on $\omega$-free
PAC fields and Granger on vector spaces with bilinear forms furnished
examples of non-simple theories for which there are nonetheless independence
relations satisfying many of the fundamental properties of non-forking
independence in simple theories. These properties include extension,
symmetry, and the independence theorem. Chernikov and the second-named
author proved an analogue of one direction of the Kim-Pillay theorem
for NSOP$_{1}$ theories, showing essentially that the existence of
an independence relation with these properties implies that a theory
is NSOP$_{1}$. To establish the other direction, the first and second-named
authors introduced \emph{Kim-independence} and showed that it is well-behaved
in any NSOP$_{1}$ theory. The theory of Kim-independence provides
an explanation for the simplicity-like phenomena observed in certain
non-simple examples and a central issue of research concerning NSOP$_{1}$
theories is to determine the extent to which properties of non-forking
independence in simple theories carry over to Kim-independence in
NSOP$_{1}$ theories. This paper addresses the specific issue of local
character for Kim-independence.

Simple theories are defined to be the theories in which forking satisfies
local character. Local character of non-forking asserts that there
is some cardinal $\kappa\left(T\right)$ so that, for any complete
type $p$ over $A$, there is a set $B\subseteq A$ with $\left|B\right|<\kappa\left(T\right)$
over which $p$ does not fork. An analogue of local character for
Kim-independence in NSOP$_{1}$ theories was proved by the first-
and second-named authors in \cite[Theorem 4.5]{kaplan2017kim}. It
was shown there that if $T$ is NSOP$_{1}$ and $M\models T$, then
for any $p\in S\left(M\right)$, there is $N\prec M$ with $|N|<\kappa=\left(2^{\left|T\right|}\right)^{+}$
such that $p$ does not Kim-fork over $N$. 

However, this result was an unsatisfactory generalization of local
character in simple theories for three reasons. First, with respect
to non-forking, it follows almost immediately that if $\kappa\left(T\right)$
exists at all, it can be taken to be $\left|T\right|^{+}$: given
a type $p\in S\left(A\right)$ with no $B\subseteq A$ of size $<\left|T\right|^{+}$
over which $p$ does not fork, one can find a chain of forking types
of length $\left|T\right|^{+}$ and then by the pigeonhole principle,
some formula must fork infinitely often with respect to the same disjunction
of dividing formulas. This equivalence is no longer immediate when
considering Kim-independence, because of the added constraint that
the formulas must divide with respect to Morley sequences and it was
asked \cite[Question 4.7]{kaplan2017kim} if $\left(2^{\left|T\right|}\right)^{+}$
can be replaced by $\left|T\right|^{+}$ in an arbitrary NSOP$_{1}$
theory. Secondly, non-forking independence satisfies \emph{base monotonicity},
which means that if $p\in S\left(A\right)$ does not fork over $B$,
then $p$ does not fork over $B'$ whenever $A\subseteq B'\subseteq B$.
In other words, local character of forking implies that every type
does not fork over an entire \emph{cone} of small subsets of its domain.
However, in an NSOP$_{1}$ theory $T$, Kim-independence satisfies
base monotonicity if and only if $T$ is simple. One would like an
analogue of local character for NSOP$_{1}$ theories that shows that
types over models do not Kim-divide over \emph{many} small submodels.
Finally, local character of non-forking independence \emph{characterizes}
simple theories. Many tameness properties of Kim-independence are
known to characterize NSOP$_{1}$ theories, e.g. symmetry and the
independence theorem, so it is natural to ask if local character does
as well.

Our main theorem is:
\begin{thm}
\label{thm:Main-intro}Suppose $T$ is a complete theory with monster
model $\mathbb{M}\models T$. The following are equivalent: 
\begin{enumerate}
\item $T$ is NSOP$_{1}$. 
\item There is no continuous increasing sequence of $|T|$-sized models
$\sequence{M_{i}}{i<\left|T\right|^{+}}$ with union $M$ and $p\in S\left(M\right)$
such that $p\upharpoonright M_{i+1}$ Kim-forks over $M_{i}$ for
all $i<|T|^{+}$. 
\item For any $M\models T$, $p\in S\left(M\right)$, the set of elementary
substructures of $M$ of size $\left|T\right|$ over which $p$ does
not Kim-divide is a stationary subset of $\left[M\right]^{\left|T\right|}$. 
\item For any $M\models T$, $p\in S\left(M\right)$, the set of elementary
substructures of $M$ of size $\left|T\right|$ over which $p$ does
not Kim-divide contains a club subset of $\left[M\right]^{\left|T\right|}$. 
\item For any $M\models T$, $p\in S\left(M\right)$, the set of elementary
substructures of $M$ of size $\left|T\right|$ over which $p$ does
not Kim-divide is a club subset of $\left[M\right]^{\left|T\right|}$. 
\item Suppose that $N\models T$, $M\prec N$ and $p\in S\left(N\right)$
does not Kim-divide over $M$. Then the set of elementary substructures
of $M$ of size $\left|T\right|$ over which $p$ does not Kim-divide
is a club subset of $\left[M\right]^{\left|T\right|}$. 
\end{enumerate}
\end{thm}

The equivalence of (1) and (2) was noted in \cite[Corollary 4.6]{kaplan2017kim}
with $\left|T\right|^{+}$ replaced by $\left(2^{\left|T\right|}\right)^{+}$,
which is considerably weaker than the theorem proved here.

In particular, this theorem implies that if $T$ is NSOP$_{1}$, $M\models T$,
and $p\in S\left(M\right)$, then the set of $N\prec M$ with $\left|N\right|=\left|T\right|$
such that $p$ does not Kim-fork over $N$ is \emph{non-empty}, answering
a question asked by the first and second-named authors \cite[Question 4.7]{kaplan2017kim}.
However, by demanding a stronger form of local character, we obtain
a new characterization of NSOP$_{1}$. 
\begin{rem}
In the first draft of this paper, published online on July 2017, we
did not yet have (5) or (6) above. Shortly after that draft was available,
Pierre Simon have found an easier proof of (1) implies (4), and we
thank him for allowing us to include his proof here. Later we found
a proof of (6). These proofs uses symmetry of Kim-independence, but
are not straightforward as in the proof in simple theories, and our
original proof.

Our original proof assumes towards contradiction that local character
fails and reaches a contradiction to NSOP$_{1}$ as is done in e.g.
simple theories. For this approach to work we used stationary logic.
This logic expands first-order logic by introducing a quantifier $\textsf{aa}$
interpreted so that $M\models\left(\aa S\right)\varphi\left(S\right)$
if and only if the set of countable subsets $X\subseteq M$ such that
$M$, when expanded with the predicate $S$ interpreted as $X$, satisfies
$\varphi\left(S\right)$ contains a club of $\left[M\right]^{\omega}$.
This logic was introduced by the third-named author in \cite{shelah1975generalized}
and later studied by Mekler and the third-named author \cite{mekler1986stationary}
who showed that the satisfiability of a theory in $L\left(\aa\right)$
implies the satisfiability of a theory in a related logic, where the
second-order quantifiers range over \emph{uncountable} sets of a certain
size. This theorem, which may be regarded as a version of the upward
Lowenheim-Skolem theorem, provides a tool for ``stretching\textquotedbl{}
a family of counterexamples to local character in such a way that
preserves the cardinality and continuity constraints needed to produce
SOP$_{1}$.

After further review, we noticed that our original proof gives rise
to a new phenomenon, which we call dual local character. 

In light of all this, we decided to re-arrange the paper in the following
way. After a short preliminaries section, we prove the main theorem.
In Section \ref{sec:A proof using stationary logic} we discuss stationary
logic and describe our original proof%
.  In Section \ref{sec:Dual-local-character} we discuss the dual
local character. %
{} 
\end{rem}

\section{Preliminaries}

\subsection{NSOP$_{1}$ theories, invariant types, and Morley sequences}
\begin{defn}
\cite[Definition 2.2]{dvzamonja2004maximality} A formula $\varphi\left(x;y\right)$
has the $1$-\emph{strong order property} \emph{(SOP}$_{1})$ if there
is a tree of tuples $\sequence{a_{\eta}}{\eta\in2^{<\omega}}$ so
that 
\begin{itemize}
\item For all $\eta\in2^{\omega}$, the partial type $\set{\varphi\left(x;a_{\eta\restriction n}\right)}{n<\omega}$
is consistent. 
\item For all $\nu,\eta\in2^{<\omega}$, if $\nu\frown\langle0\rangle\unlhd\eta$
then $\left\{ \varphi\left(x;a_{\eta}\right),\varphi\left(x;a_{\nu\frown\langle1\rangle}\right)\right\} $
is inconsistent. 
\end{itemize}
\end{defn}

A theory $T$ is \emph{NSOP}$_{1}$ if no formula has SOP$_{1}$ modulo
$T$. 
\begin{fact}
\label{karyversion} \cite[Proposition 2.4]{kaplan2017kim} $T$ has
NSOP$_{1}$ if and only if there is a formula $\varphi\left(x;y\right)$,
$k<\omega$, and a sequence $\sequence{\bar{c}_{i}}{i\in I}$ with
$\overline{c}_{i}=\left(c_{i,0},c_{i,1}\right)$ satisfying: 
\begin{enumerate}
\item For all $i\in I$, $c_{i,0}\equiv_{\overline{c}_{<i}}c_{i,1}$. 
\item $\set{\varphi\left(x;c_{i,0}\right)}{i\in I}$ is consistent. 
\item $\set{\varphi\left(x;c_{i,1}\right)}{i\in I}$ is $k$-inconsistent. 
\end{enumerate}
\end{fact}

We also use following notation. Write $a\ind_{M}^{u}B$ for $\tp\left(a/MB\right)$
is finitely satisfiable in $M$, in other words it is a \emph{coheir}
of its restriction to $M$. A type $p\in S\left(M\right)$ is an \emph{heir}
of its restriction to $N\prec M$ if for every formula $\varphi\left(x;y\right)\in L\left(N\right)$
and every $b\in M$, if $\varphi\left(x;b\right)\in p$ then $\varphi\left(x;b'\right)\in p$
for some $b'\in N$. We denote this by $c\ind_{N}^{h}M$. This is
equivalent to saying that $M\ind_{N}^{u}c$. 
\begin{defn}
A global type $q\in S\left(\Mm\right)$ is called $A$\emph{-invariant}
if $b\equiv_{A}b'$ implies $\varphi\left(x;b\right)\in q$ if and
only if $\varphi\left(x;b'\right)\in q$. A global type $q$ is \emph{invariant}
if there is some small set $A$ such that $q$ is $A$-invariant.
If $q\left(x\right)$ and $r\left(y\right)$ are $A$-invariant global
types, then the type $\left(q\otimes r\right)\left(x,y\right)$ is
defined to be $\text{tp}\left(a,b/\mathbb{M}\right)$ for any $b\models r$
and $a\models q|_{\mathbb{M}b}$. It is also $A$-invariant. We define
$q^{\otimes n}\left(x_{0},\ldots,x_{n-1}\right)$ by induction: $q^{\otimes1}=q$
and $q^{\otimes n+1}=q\left(x_{n}\right)\otimes q^{\otimes n}\left(x_{0},\ldots,x_{n-1}\right)$.
\end{defn}

\begin{fact}
\cite[Lemma 4.1]{shelah1990classification}\label{average} If $T$
is any complete theory, $M\models T$, and $p\in S\left(M\right)$,
then there is a complete global type $q$ extending $p$ which is,
moreover, finitely satisfiable in $M$. In particular, $q$ is $M$-invariant. 
\end{fact}

\begin{defn}
Suppose $q$ is an $A$-invariant global type and $I$ is a linearly
ordered set. By a \emph{Morley sequence in }$q$ \emph{over} $A$
\emph{of order type} $I$, we mean a sequence $\sequence{b_{\alpha}}{\alpha\in I}$
such that for each $\alpha\in I$, $b_{\alpha}\models q|_{Ab_{<\alpha}}$
where $b_{<\alpha}=\sequence{b_{\beta}}{\beta<\alpha}$. Given a linear
order $I$, we will write $q^{\otimes I}$ for the $A$-invariant
type in variables $\sequence{x_{\alpha}}{\alpha<I}$ so that for any
$B\supseteq A$, if $\overline{b}\models q^{\otimes I}|_{B}$ then
$b_{\alpha}\models q|_{Bb_{<\alpha}}$ for all $\alpha\in I$. If
$q$ is, moreover, finitely satisfiable in $A$, then we refer to
a Morley sequence in $q$ over $A$ as a \emph{coheir sequence} over
$A$.
\end{defn}

The above definition of $q^{\otimes I}$ generalizes the finite tensor
product $q^{\otimes n}$ \textendash{} given any global $A$-invariant
type $q$ and linearly ordered set $I$, one may easily show that
$q^{\otimes I}$ exists and is $A$-invariant by compactness.
\begin{defn}
Suppose $M$ is a model. 
\begin{enumerate}
\item Given a formula $\varphi\left(x;b\right)$ and a global $M$-invariant
type $q\supseteq\text{tp}\left(b/M\right)$, we say that $\varphi\left(x;b\right)$
\emph{$k$-Kim-divides over }$M$\emph{ via }$q$ if, whenever $\sequence{b_{i}}{i<\omega}$
is a Morley sequence over $M$ in $q$, then $\set{\varphi\left(x;b_{i}\right)}{i<\omega}$
is $k$-inconsistent. 
\item If $q$ is a global $M$-invariant type with $q\supseteq\text{tp}\left(b/M\right)$,
we say $\varphi\left(x;b\right)$ \emph{Kim-divides over} $M$ \emph{via}
$q$ if $\varphi\left(x;b\right)$ $k$-Kim-divides over $M$ via
$q$ for some $k<\omega$.
\item We say $\varphi\left(x;b\right)$ Kim-divides over $M$ if $\varphi\left(x;b\right)$
Kim-divides over $M$ via $q$ for some global $M$-invariant $q\supseteq\text{tp}\left(b/M\right)$. 
\item We say that $\varphi\left(x;b\right)$ Kim-forks over $M$ if it implies
a finite disjunction of formulas, each Kim-dividing over $M$. 
\item We write $a\ind_{M}^{K}B$ for $\tp\left(a/MB\right)$ does not Kim-fork
(or Kim-independent) over $M$. 
\end{enumerate}
\end{defn}

Note that if $a\ind_{M}^{u}B$ then $a\ind_{M}^{f}B$ (i.e. $\tp\left(a/BM\right)$
does not fork over $M$) which implies $a\ind_{M}^{K}B$. 
\begin{fact}
\cite[Theorem 3.15]{kaplan2017kim} \label{kimslemma} The following
are equivalent for the complete theory $T$: 
\begin{enumerate}
\item $T$ is NSOP$_{1}$. 
\item (Kim's lemma for Kim-dividing) Given any model $M\models T$ and formula
$\varphi\left(x;b\right)$, $\varphi\left(x;b\right)$ Kim-divides
via $q$ for \emph{some} global $M$-invariant $q\supseteq\text{tp}\left(b/M\right)$
if and only if $\varphi\left(x;b\right)$ Kim-divides via $q$ for
\emph{all} global $M$-invariant $q\supseteq\text{tp}\left(b/M\right)$. 
\end{enumerate}
\end{fact}

From this it easily follows that Kim-forking is equal to Kim-dividing
\cite[Proposition 3.19]{kaplan2017kim}. The notion of Kim independence,
denoted by $\ind^{K}$, satisfies many nice properties which turn
out to be equivalent to NSOP$_{1}$. 
\begin{fact}
\cite[Theorem 8.1]{kaplan2017kim} The following are equivalent for
the complete theory $T$: 
\begin{enumerate}
\item $T$ is NSOP$_{1}$.
\item Symmetry of Kim independence over models: $a\ind_{M}^{K}b$ iff $b\ind_{M}^{K}a$
for any $M\models T$.
\item Independence theorem over models: if $A\ind_{M}B$, $c\ind_{M}A$,
$c'\ind_{M}B$ and $c\equiv_{M}c'$ then there is some $c''\ind_{M}AB$
such that $c''\equiv_{MA}c$ and $c''\equiv c_{MB}'$.
\end{enumerate}
\end{fact}

\begin{fact}
\label{fact:Kim Morley is consistent}\cite[Lemma 7.6]{kaplan2017kim}Suppose
that $T$ is NSOP$_{1}$ and that $\sequence{a_{i}}{i<\omega}$ is
an $\ind^{K}$-Morley sequence over $M$ in the sense that $a_{i}\ind_{M}^{K}a_{<i}$
and the sequence is indiscernible. Then if $\varphi\left(x,a_{0}\right)$
does not Kim-divide over $M$, then $\set{\varphi\left(x,a_{i}\right)}{i<\omega}$
does not Kim-divide over $M$, and in particular it is consistent. 
\end{fact}

\subsection{The generalized club filter}
\begin{defn}
\label{clubdef} Let $\kappa$ be a cardinal and $X$ a set with $\left|X\right|\geq\kappa$.
We write $\left[X\right]^{\kappa}$ to denote $\set{Y\subseteq X}{\left|Y\right|=\kappa}$. 
\begin{enumerate}
\item A set $C\subseteq\left[X\right]^{\kappa}$ is \emph{unbounded} if
for every $Y\in\left[X\right]^{\kappa}$, there is some $Z\in C$
with $Y\subseteq Z$. 
\item A set $C\subseteq\left[X\right]^{\kappa}$ is \emph{closed} if, whenever
$\sequence{Y_{i}}{i<\alpha\leq\kappa}$ is a chain in $C$, i.e. each
$Y_{i}\in C$ and $i<j<\alpha$ implies $Y_{i}\subseteq Y_{j}$, then
$\bigcup_{i<\alpha}Y_{i}\in C$. 
\item A set $C\subseteq\left[X\right]^{\kappa}$ is \emph{club} if it is
closed and unbounded. 
\item A set $S\subseteq\left[X\right]^{\kappa}$ is \emph{stationary} if
$S\cap C\neq\emptyset$ for every club $C\subseteq\left[X\right]^{\kappa}$. 
\end{enumerate}
The \emph{club filter} on $\left[X\right]^{\kappa}$ is the filter
generated by the clubs. If $\left|X\right|=\kappa$, then the club
filter on $\left[X\right]^{\kappa}$ is the principal ultrafilter
consisting of subsets of $\left[X\right]^{\kappa}$ containing $X$. 
\end{defn}

\begin{example}
\label{exa:elementary substructures}If $M$ is an $L$-structure
of size $\geq\kappa\geq\left|L\right|$, then the collection of elementary
substructures of $M$ of size $\kappa$ is a club in $\left[M\right]^{\kappa}$. 
\end{example}

\begin{rem}
In the literature, e.g. \cite[Definition 8.21]{jech2013set}, the
above definitions are given instead for subsets of $\mathcal{P}_{\kappa^{+}}\left(X\right)=\set{Y\subseteq X}{\left|Y\right|<\kappa^{+}}$
but note that $\left[X\right]^{\kappa}$ is a club subset of $\mathcal{P}_{\kappa^{+}}\left(X\right)$,
hence all definitions relativize to this set in the natural way. 
\end{rem}

\begin{fact}
\label{clubfacts} Let $\kappa$ be a cardinal and $X$ a set with
$\left|X\right|\geq\kappa^{+}$. 
\begin{enumerate}
\item The club filter on $\left[X\right]^{\kappa}$ is $\kappa^{+}$-complete
\cite[Theorem 8.22]{jech2013set}. 
\item For every club $C\subseteq\left[X\right]^{\kappa}$, there is a collection
of finitary functions $\overline{f}=\sequence{f_{i}}{i<\kappa}$ with
$f_{i}:X^{n_{i}}\to X$ such that $C_{\overline{f}}:=\set{Y\in\left[X\right]^{\kappa}}{f_{i}\left(Y^{n_{i}}\right)\subseteq Y\text{ for all }i<\kappa}$
is contained in $C$. Equivalently, there is a function $F:X^{<\omega}\to\left[X\right]^{\kappa}$
such that $C_{F}\subseteq C$ \cite[Lemma 8.26]{jech2013set}. 
\item Conversely, given a collection of finitary functions $\overline{f}=\sequence{f_{i}}{i<\kappa}$
with $f_{i}:X^{n_{i}}\to X$, the set $C_{\overline{f}}$ is club
in $\left[X\right]^{\kappa}$. 
\item When $\kappa=\omega$, for any club $C\subseteq\left[X\right]^{\kappa}$,
there is a function $F:X^{<\omega}\to X$ such that $C_{F}\subseteq C$
\cite[Theorem 8.28]{jech2013set}. 
\end{enumerate}
\end{fact}

We leave the proof of the next lemma to the reader.
\begin{lem}
\label{niceclubs} Suppose $\lambda$ is a cardinal, $X$ is a set
with $\left|X\right|=\lambda^{+}$, and $\sequence{Y_{\alpha}}{\alpha<\lambda^{+}}$
is an increasing continuous sequence of sets of cardinality $\lambda$
with union $X$. Then $\set{Y_{\alpha}}{\alpha<\lambda^{+}}$ is a
club of $\left[X\right]^{\lambda}$. In particular, if $X=\lambda^{+}$
and $C\subseteq\lambda^{+}\setminus\lambda$ is a club of $\lambda^{+}$,
then $C$ is a club of $\left[X\right]^{\lambda}$.
\end{lem}

\section{\label{sec:Pierre's proof} proof of Theorem \ref{thm:Main-intro}}

\subsection{A short proof of (1) implies (4) in Theorem \ref{thm:Main-intro}
using heirs}

Here we give a short proof of (1) implies (4) in Theorem \ref{thm:Main-intro},
due to Pierre Simon. We thank him for allowing us to include this
proof. 
\begin{lem}
\label{getting heirs} Suppose $p\left(x\right)\in S\left(M\right)$,
$M\models T$. Then the set of $N\prec M$ such that $\left|N\right|=\left|T\right|$
and $p$ is an heir of $p|_{N}$ is a club subset of $\left[M\right]^{\left|T\right|}$. 
\end{lem}

\begin{proof}
It is easy to verify that this set is closed under increasing unions,
so it is enough to show that it contains a club. 

Consider the $L_{p}$-structure $M_{p}$ expanding $M$ by forcing
$p$ to be definable \textemdash{} i.e. for every $L$-formula $\varphi\left(x;y\right)$
add a relation $R_{\varphi}\left(y\right)$ interpreted as $\set{b\in M^{\left|y\right|}}{\varphi\left(x,b\right)\in p}$.
Note that $\left|L_{p}\right|=\left|L\right|$. Then if $N'\prec M_{p}$
then its $L$-reduct $N$ is such that $p$ is an heir of $p|_{N}$.
 Thus we are done by Example \ref{exa:elementary substructures}.
\end{proof}
\begin{thm}
Suppose $T$ is NSOP$_{1}$. If $M\models T$ and $p\in S\left(M\right)$,
then the set of elementary substructures $N\prec M$ with $\left|N\right|=\left|T\right|$
such that $p$ does not Kim-divide over $N$ contains a club. 
\end{thm}

\begin{proof}
By Lemma \ref{getting heirs}, it suffices to show that if $p$ is
an heir of $p|_{N}$, then $p$ does not Kim-divide over $N$. But
if $p$ is an heir of $p|_{N}$, then, given $c\models p$, $M\ind_{N}^{u}c$,
hence $M\ind_{N}^{K}c$ by symmetry of Kim-independence (in fact one
needs only a weak version of symmetry, see \cite[Proposition 3.22]{kaplan2017kim})
which implies $c\ind_{N}^{K}M$. This shows that $p$ does not Kim-divide
over $N$. 
\end{proof}

\subsection{A proof of (1) implies (6) in Theorem \ref{thm:Main-intro}}
\begin{lem}
\label{clubfinding} Suppose $T$ is an arbitrary theory and $M\models T$
with $\left|M\right|\geq\left|T\right|=\kappa$. Given any global
$M$-finitely satisfiable type $q$, let $C_{q}$ denote the set of
$N\prec M$ with $\left|N\right|=\kappa$ such that $q^{\otimes\omega}|_{N}=r^{\otimes\omega}|_{N}$
for some global $N$-finitely satisfiable type $r$. Then: 
\begin{enumerate}
\item $C_{q}$ is in the club filter on $\left[M\right]^{\kappa}$. 
\item Given any set $A$, there is some $N\prec M$ of size $\leq\left|T\right|+\left|A\right|$
such that $A\subseteq N$ and $q^{\otimes\omega}|_{N}$ is a type
of a Morley sequence generated by some global type $r$ finitely satisfiable
in $N$ and if $\varphi\left(x,c\right)$ Kim-divides over $M$ via
$q$ then $\varphi\left(x,c\right)$ Kim-divides over $N$ via $r$. 
\end{enumerate}
\end{lem}

\begin{proof}
One proof of (1) essentially follows from the proof of \cite[Lemma 4.4]{kaplan2017kim},
so we also give an alternative one. Let $\bar{a}=\sequence{a_{i}}{i<\omega}$
be a coheir sequence generated by $q$ over $M$. Then, $N\in C_{q}$
iff $N\prec M$ and $\bar{a}$ is a coheir sequence over $N$ in the
sense that $\tp\left(a_{i}/a_{<i}N\right)$ if finitely satisfiable
in $N$. Thus it is easy to see that $C_{q}$ is closed under unions. 

Note that if $N\prec M$ is such that $\tp\left(\bar{a}/M\right)$
is an heir extension of its restriction to $N$, then $N\in C_{q}$:
if $\varphi\left(a_{i},a_{<i}\right)$ holds when $\varphi\left(x,y\right)$
is some formula over $N$, then for some $c\in M$, $\varphi\left(c,a_{<i}\right)$
holds, and by choice of $N$, we may assume that $c\in N$. Now Lemma
\ref{getting heirs} finishes the proof. 

(2) is immediate from (1), applied to the theory $T\left(A\right)$
obtained from $T$ by adding constants for the elements of $A$.
\end{proof}
\begin{thm}
\label{clubinbase} Suppose $T$ is NSOP$_{1}$ with $|T|=\kappa$
and $M\models T$. Then for a finite tuple $b$ and any set $A$,
the following are equivalent: 
\begin{enumerate}
\item $A\ind_{M}^{K}b$. 
\item There is a club $C\subseteq\left[M\right]{}^{\kappa}$ of elementary
substructures of $M$ such that $A\ind_{N}^{K}b$ for all $N\in C$. 
\item There is a stationary set $S\subseteq\left[M\right]{}^{\kappa}$ of
elementary substructures of $M$ such that $A\ind_{N}^{K}b$ for all
$N\in S$. 
\end{enumerate}
\end{thm}

\begin{proof}
(1)$\implies$(2) Suppose that $A\ind_{M}^{K}b$. Let $q\supseteq\text{tp}\left(b/M\right)$
be a global $M$-finitely satisfiable type and choose $\sequence{b_{i}}{i<\omega}\models q^{\otimes\omega}|_{M}$
with $b_{0}=b$. By Lemma \ref{clubfinding}, there is a club $C_{q}$
of elementary substructures $N\prec M$ with $|N|=|T|$ so that $q^{\otimes\omega}|_{N}=r^{\otimes\omega}|_{N}$
for some global $N$-finitely satisfiable type $r$. Fix $N\in C$,
$a$ a finite tuple from $A$ and $\varphi\left(x;b,n\right)\in\text{tp}\left(a/Nb\right)$.
As $a\ind_{M}^{K}b$, we know $\set{\varphi\left(x;b_{i},n\right)}{i<\omega}$
is consistent. As $\sequence{b_{i}}{i<\omega}$ is also a Morley sequence
over $N$ in a global $N$-finitely satisfiable type, it follows from
Kim's lemma for Kim-dividing (Fact \ref{kimslemma}) that $\varphi\left(x;b,n\right)$
does not Kim-divide over $N$. As $\varphi\left(x;b,n\right)$ was
arbitrary, we conclude $a\ind_{N}^{K}b$. Since this was true for
any $a$, we have that $A\ind_{N}^{K}b$. 

(2)$\implies$(3) is immediate.

(3)$\implies$(1) Suppose $a\nind_{M}^{K}b$ for some finite tuple
$a$ from $A$. Let $\varphi\left(x;b,m\right)\in\text{tp}\left(a/Mb\right)$
be a formula witnessing this. Fix $q\supseteq\text{tp}\left(b/M\right)$
a global $M$-finitely satisfiable type and $\sequence{b_{i}}{i<\omega}\models q^{\otimes\omega}|_{M}$.
Let $C'=\set{N\prec M}{\left|N\right|=\left|T\right|\text{ and }m\in N}$.
The set $C'$ is clearly club so the intersection $C''=C_{q}\cap C'$
is in the club filter on $\left[M\right]^{\kappa}$. If $N\in C''$
and $q^{\otimes\omega}|_{N}=r^{\otimes\omega}|_{N}$ for some global
type $r$ finitely satisfiable in $N$, then $\varphi\left(x;b,m\right)\in\text{tp}\left(a/Nb\right)$
and $\sequence{b_{i}}{i<\omega}$ realizes $r^{\otimes\omega}|_{N}$.
As $\set{\varphi\left(x;b_{i},m\right)}{i<\omega}$ is inconsistent,
we have $a\nind_{N}^{K}b$. As $S$ is stationary, it must intersect
$C''$, so we get a contradiction. 
\end{proof}
\begin{cor}
\label{cor:club in base infinite} Suppose $T$ is NSOP$_{1}$ with
$|T|=\kappa$ and $M\models T$. Then for a finite tuple $a$ and
any set $B$, the following are equivalent: 
\begin{enumerate}
\item $a\ind_{M}^{K}B$. 
\item There is a club $C\subseteq\left[M\right]{}^{\kappa}$ of elementary
substructures of $M$ such that $a\ind_{N}^{K}B$ for all $N\in C$. 
\item There is a stationary set $S\subseteq\left[M\right]{}^{\kappa}$ of
elementary substructures of $M$ such that $a\ind_{N}^{K}B$ for all
$N\in S$.
\end{enumerate}
\end{cor}

\begin{proof}
Follows immediately from symmetry of Kim-independence and Theorem
\ref{clubinbase}. 
\end{proof}
\begin{lem}
\label{lem:if phi(x,a) forks over big then also over small}Suppose
T is NSOP$_{1}$.  Assume $M\prec N$. Suppose that $a\ind_{M}^{K}N$
and $\varphi\left(x,a\right)$ Kim-divides over $N$ for $\varphi\left(x,y\right)\in L\left(M\right)$.
Then $\varphi\left(x,a\right)$ Kim-divides over $M$.
\end{lem}

\begin{proof}
Let $\sequence{a_{i}}{i<\omega}$ be an indiscernible sequence over
$N$ starting with $a_{0}=a$ such that $a_{i}\ind_{N}^{h}a_{<i}$
and $\set{\varphi\left(x,a_{i}\right)}{i<\omega}$ is inconsistent
(to construct it, let $\sequence{b_{i}}{i\in\Zz}$ be a coheir sequence
in the type of $\tp\left(a/N\right)$, so in particular $b_{i}\ind_{N}^{u}b_{<i}$
for $i<0$, hence $b_{>i}\ind_{N}^{u}b_{i}$ by transitivity of $\ind^{u}$,
and let $a_{i}=b_{-i}$ for $i<\omega$). 

Then $\sequence{a_{i}}{i<\omega}$ is an $\ind^{K}$-Morley sequence
over $M$ in the sense that $a_{i}\ind_{M}^{K}a_{<i}$. To see this,
suppose not, i.e., by symmetry suppose that $a_{<i}\nind_{M}^{K}a_{i}$.
Then for some formula $\psi\left(z,x\right)$ over $M$, $\psi\left(a_{<i},a_{i}\right)$
holds and $\psi\left(z,a_{i}\right)$ Kim-divides over $M$. Since
$a_{<i}\ind_{N}^{u}a_{i}$, for some $n\in N$, $\psi\left(n,a_{i}\right)$
holds. However, since $a_{i}\equiv_{N}a$, by symmetry $N\ind_{M}^{K}a_{i}$
\textemdash{} contradiction. 

Suppose that $\varphi\left(x,a\right)$ does not Kim-divide over $M$.
Then by Fact \ref{fact:Kim Morley is consistent}, $\set{\varphi\left(x,a_{i}\right)}{i<\omega}$
is consistent \textemdash{} contradiction. 
\end{proof}
\begin{lem}
\label{lem:closed under union} Suppose T is NSOP$_{1}$. Suppose
that $\sequence{M_{i}}{i\leq\alpha}$ is an increasing sequence of
elementary substructures of a model $N$, that $M_{\alpha}=\bigcup\set{M_{i}}{i<\alpha}$
and that $p\in S\left(N\right)$. Assume that $p$ does not Kim-fork
over $M_{i}$ for all $i<\alpha$. Then $p$ does not Kim-fork over
$M_{\alpha}$.
\end{lem}

\begin{proof}
Let $a\models p$. We want to show that $a\ind_{M_{\alpha}}^{K}N$,
so by symmetry it is enough to show that $N\ind_{M_{\alpha}}^{K}a$.
Suppose not. Then there is some formula $\varphi\left(x,y\right)$
in $L\left(M_{\alpha}\right)$ and some $b\in N$ such that $\varphi\left(b,a\right)$
holds and $\varphi\left(x,a\right)$ Kim-divides over $M_{\alpha}$.
Let $i<\alpha$ be such that $\varphi\left(x,y\right)\in L\left(M_{i}\right)$.
Since $M_{\alpha}\subseteq N$ and $a\ind_{M_{i}}^{K}N$ by assumption,
$a\ind_{M_{i}}^{K}M_{\alpha}$. Hence by Lemma \ref{lem:if phi(x,a) forks over big then also over small},
$\varphi\left(x,a\right)$ Kim-divides over $M_{i}$. Hence $b\nind_{M_{i}}^{K}a$.
But this is a contradiction since $a\ind_{M_{i}}^{K}N$ so by symmetry
$b\ind_{M_{i}}^{K}a$. 
\end{proof}
We can now prove (1) $\implies$ (6) from Theorem \ref{thm:Main-intro}. 
\begin{thm}
\label{thm:(1) implies (6)}Suppose that $T$ is NSOP$_{1}$. Suppose
that $a$ is a finite tuple, $a\ind_{M}^{K}N$ and $M\prec N$. Then
the set $E$ of $M'\in\left[M\right]^{\left|T\right|}$ such that
$M'\prec M$ and $a\ind_{M'}^{K}N$ is a club. 
\end{thm}

\begin{proof}
The family $E$ is closed under unions by Lemma \ref{lem:closed under union}.
Hence to finish we only need to show that $E$ contains a club, and
this follows from Corollary \ref{cor:club in base infinite} (1) $\implies$
(2). 
\end{proof}

\subsection{The equivalence (1)\textendash (6)}

We finish the proof of Theorem \ref{thm:Main-intro} with the following. 
\begin{thm}
\label{mainthm(1)-(6)} Suppose $T$ is a complete theory. The following
are equivalent: 
\begin{enumerate}
\item $T$ is NSOP$_{1}$. 
\item There is no continuous increasing sequence of $|T|$-sized models
$\sequence{M_{i}}{i<\left|T\right|^{+}}$ with union $M$ and $p\in S\left(M\right)$
such that $p\upharpoonright M_{i+1}$ Kim-forks over $M_{i}$ for
all $i<|T|^{+}$. 
\item For any $M\models T$, $p\in S\left(M\right)$, the set of elementary
substructures of $M$ of size $\left|T\right|$ over which $p$ does
not Kim-divide is a stationary subset of $\left[M\right]^{\left|T\right|}$. 
\item For any $M\models T$, $p\in S\left(M\right)$, the set of elementary
substructures of $M$ of size $\left|T\right|$ over which $p$ does
not Kim-divide contains a club subset of $\left[M\right]^{\left|T\right|}$. 
\item For any $M\models T$, $p\in S\left(M\right)$, the set of elementary
substructures of $M$ of size $\left|T\right|$ over which $p$ does
not Kim-divide is a club subset of $\left[M\right]^{\left|T\right|}$. 
\item Suppose that $N\models T$, $M\prec N$ and $p\in S\left(N\right)$
does not Kim-divide over $M$. Then the set of elementary substructures
of $M$ of size $\left|T\right|$ over which $p$ does not Kim-divide
is a club subset of $\left[M\right]^{\left|T\right|}$. 
\end{enumerate}
\end{thm}

\begin{proof}
(1)$\implies$(6) is Theorem \ref{thm:(1) implies (6)}.

(6) $\implies$ (5)$\implies$(4)$\implies$(3) is trivial (for (6)
implies (5), note that for $p\in S\left(M\right)$, $p$ does not
Kim-divide over $M$ trivially). 

(3)$\implies$(2) By Lemma \ref{niceclubs}, $C=\set{M_{i}}{i<\left|T\right|^{+}}$
is a club of $\left[M\right]^{\left|T\right|}$. As $T$ is NSOP$_{1}$,
there is a stationary set $S\subseteq\left[M\right]^{\left|T\right|}$
such that $N\in S$ implies $p$ does not Kim-fork over $N$. Choose
any $M_{i}\in C\cap S$ to obtain a contradiction.

(2)$\implies$(1). Suppose $T$ has SOP$_{1}$ as witnessed by some
formula $\varphi\left(x,y\right)$. Let $T^{sk}$ be a Skolemized
expansion of $T$. Then $T^{sk}$ also has SOP$_{1}$ as witnessed
by $\varphi\left(x,y\right)$. Thus by Proposition \ref{karyversion},
we can find a formula $\varphi\left(x,y\right)$ and an array $\sequence{c_{i,j}}{i<\omega,j<2}$
such that $c_{i,0}\equiv_{\overline{c}_{<i}}c_{i,1}$ for all $i<\omega$,
$\set{\varphi\left(x,c_{i,0}\right)}{i<\omega}$ is consistent and
$\set{\varphi\left(x;c_{i,1}\right)}{i<\omega}$ is 2-inconsistent
(all in $\mathbb{M}^{sk}$). By Ramsey and compactness we may assume
that $\sequence{\overline{c}_{i}}{i<\omega}$ is indiscernible (with
respect to $\mathbb{M}^{sk}$) and extend this sequence to length
$\left|T\right|^{+}$.

For $i\leq|T|^{+}$, let $N_{i}=\text{dcl}\left(\overline{c}_{<i}\right)$
(in $\mathbb{M}^{sk}$). Then for every limit ordinal $\delta<|T|^{+}$,
$\varphi\left(x,c_{\delta,1}\right)$ Kim-divides over $N_{\delta}$
as the sequence $\sequence{c_{j,1}}{\delta\leq j<\left|T\right|^{+}}$
is indiscernible and for all $\delta\leq j$, $\overline{c}_{j}\ind_{N_{\delta}}^{u}\overline{c}_{>j}$.
As $c_{\delta,1}\equiv_{\bar{c}_{<\delta}}c_{\delta,0}$, it follows
that $c_{\delta,1}\equiv_{N_{\delta}}c_{\delta,0}$, and hence $\varphi\left(x,c_{\delta,0}\right)$
also Kim-divides. Let $p\in S\left(N_{\left|T\right|^{+}}\right)$
be any complete type containing $\set{\varphi\left(x,c_{\delta,0}\right)}{\delta<\kappa}$,
which is possible as this partial type is consistent. The sequence
$\sequence{N_{\delta}}{\delta\in\text{lim}\left(\left|T\right|^{+}\right)}$
is an increasing and continuous sequence of elementary substructures
of $N_{\left|T\right|^{+}}$ of size $\left|T\right|$ with union
$N_{\left|T\right|^{+}}$ witnessing that (2) fails.
\end{proof}
\begin{cor}
\label{bigger sets} Suppose $T$ is NSOP$_{1}$, $M\models T$, $M\prec N$,
and $p\in S\left(N\right)$. Then $p$ does not Kim-fork over $M$
iff for every $\kappa$ with $\left|T\right|\leq\kappa\leq\left|M\right|$,
the set of elementary substructures of $M$ of size $\kappa$ over
which $p$ does not Kim-divide is a club subset of $\left[M\right]^{\kappa}$.
\end{cor}

\begin{proof}
Suppose that $p$ does not Kim-fork over $M$. Let $A\subseteq M$
be any subset of $M$ of size $\kappa$ and apply Theorem \ref{thm:Main-intro}
to the theory $T\left(A\right)$ obtained from $T$ by adding new
constant symbols for the elements of $A$. 

For the other direction, apply the left hand side with $\kappa=\left|T\right|$
and use Corollary \ref{cor:club in base infinite}. 
\end{proof}
\begin{cor}
\label{big objects} Suppose $T$ is NSOP$_{1}$ and $M\models T$.
Then given any set $A$, there is a club $E\subseteq\left[M\right]^{\left|T\right|+\left|A\right|}$
such that $N\in E$ iff $A\ind_{N}^{K}M$. 
\end{cor}

\begin{proof}
Let $\kappa=\left|A\right|+\left|T\right|$. By Corollary \ref{bigger sets},
we know for each finite tuple $a$ from $A$, there is a club $E_{a}\subseteq\left[M\right]^{\kappa}$
so that $N\in E_{a}$ iff $a\ind_{N}^{K}M$. Let $E=\bigcap_{a\in A}E_{a}$.
As $|A|\leq\kappa$ and the club filter on $\left[M\right]^{\kappa}$
is $\kappa^{+}$-complete (Fact \ref{clubfacts}(1)), $E$ is a club
of $\left[M\right]^{\kappa}$. By the strong finite character of Kim-independence,
we have $A\ind_{N}^{K}M$ iff $N\in E$. 
\end{proof}

\subsection{A sample application}
\begin{prop}
Suppose $T$ is NSOP$_{1}$ and $A\models T$. Given any set $C$,
there is some $C'\supseteq C$ with $\left|C'\right|=\left|C\right|+\left|T\right|$
such that $C'\cap A$ is a model and $C'\ind_{A\cap C'}^{K}A$. 
\end{prop}

\begin{proof}
Let $\kappa=\left|C\right|+\left|T\right|$. Let $C_{0}=C$ and, by
Corollary \ref{big objects}, we may let $E_{0}\subseteq\left[A\right]^{\kappa}$
be a club of elementary substructures of $A$ such that $N\in E_{0}$
implies $C_{0}\ind_{N}^{K}A$. By induction, we will choose sets $C_{i}$,
clubs $E_{i}\subseteq\left[A\right]^{\kappa}$, and models $X_{i}\prec A$
such that 
\begin{enumerate}
\item $X_{i}\in\bigcap_{j\leq i}E_{i}$ and $C_{i}\cap A\subseteq X_{i}$. 
\item $C_{i+1}=C_{i}\cup X_{i}$. 
\item For all $N\in E_{i}$, we have $C_{i}\ind_{N}^{K}A$. 
\end{enumerate}
Given $\sequence{C_{i},X_{i},E_{i}}{i\leq n}$, let $C_{n+1}=C_{n}\cup X_{n}$.
By Corollary \ref{big objects}, we may let $E_{n+1}\subseteq\left[A\right]^{\kappa}$
be a club such that $N\in E_{n+1}$ implies $C_{n+1}\ind_{N}^{K}A$.
As 
\[
\set{X\in\left[A\right]^{\kappa}}{C_{n+1}\cap A\subseteq X}
\]
 is a club of $\left[A\right]^{\kappa}$, we may choose $X_{n+1}\in\bigcap_{i\leq n+1}E_{i}$
containing $C_{n+1}\cap A$. This completes the induction.

Let $C_{\omega}=\bigcup_{i<\omega}C_{i}$. By construction, $C_{\omega}\cap A=\bigcup_{i<\omega}X_{i}$.
As $i<j$ implies $X_{i}\subseteq X_{j}$, and $i\geq n$ implies
$X_{i}\in E_{n}$, it follows that 
\[
C_{\omega}\cap A=\bigcup_{i\geq n}X_{i}\in E_{n}
\]
for all $n$, as $E_{n}$ is club. Also as each $X_{i}$ is a model,
this additionally shows that $C_{\omega}\cap A$ is a model. Moreover,
if $c\in C_{\omega}$ is a finite tuple, there is some $n$ so that
$c\in C_{n}$, hence $c\ind_{C_{\omega}\cap A}^{K}A$, by the choice
of $E_{n}$. Setting $C'=C_{\omega}$, we finish. 
\end{proof}

\subsection{Open questions}
\begin{question}
\label{que:Is-the-dual}Is the dual of Lemma \ref{lem:if phi(x,a) forks over big then also over small}
also true? Namely, suppose that $a\ind_{M}^{K}N$ and $\varphi\left(x,a\right)$
Kim-divides over $M$ for $\varphi\left(x,y\right)\in L\left(M\right)$.
Then is it true that $\varphi\left(x,a\right)$ Kim-divides over $N$?
\end{question}

If the answer to Question \ref{que:Is-the-dual} is ``yes'', then
we have the following weak form of transitivity (note that a full
version of transitivity does not hold, see \cite[Section 9.2]{kaplan2017kim}).
\begin{claim}
(Weak form of transitivity) Suppose the answer is ``yes''. Let $M\prec N$.
Suppose that $a\ind_{M}^{K}N$ and $a\ind_{N}^{K}B$. Then $a\ind_{M}^{K}B$. 
\end{claim}

\begin{proof}
Suppose not. Then by symmetry there is a formula $\varphi\left(x,y\right)$
over $M$ such that $\varphi\left(b,a\right)$ holds for some $b\in B$
and $\varphi\left(x,a\right)$ Kim-divides over $M$. However, since
$b\ind_{N}^{K}a$, $\varphi\left(x,a\right)$ does not Kim-divide
over $N$. By assumption we arrive at a contradiction. 
\end{proof}
\begin{question}
Does the weak form of transitivity hold in NSOP$_{1}$ theories?
\end{question}

\begin{question}
The proof of (1) implies (6) in Theorem \ref{thm:Main-intro} relied
heavily on symmetry of Kim-independence, whose proof assumes that
the whole theory is NSOP$_{1}$. However, a closer look at the proof
of (1) implies (4) given in Section \ref{sec:Pierre's proof}, or
observing the proof using stationary logic given below, we see that
for (1) implies (4), we only need that a particular formula $\varphi\left(x,y\right)$
does not have an SOP$_{1}$ array as in Fact \ref{karyversion}. Can
the same be said for (1) implies (6)?
\end{question}

\begin{question}
Is there a local counterpart to Lemma \ref{lem:closed under union}.
Namely, under NSOP$_{1}$, assume that $\varphi\left(x,a\right)$
does not Kim-fork over $M_{i}$ for $i<\alpha$ an increasing union.
Is it true that $\varphi\left(x,a\right)$ does not Kim-fork over
$\bigcup_{i<\alpha}M_{i}$?
\end{question}

\begin{question}
Is it true that $T$ is NSOP$_{1}$ if and only if for every $M\models T$
and complete type $p\in S\left(M\right)$, there is some $N\prec M$
of cardinality $\leq\left|T\right|$ such that $p$ does not Kim-fork
over $N$?
\end{question}

\section{\label{sec:A proof using stationary logic}A proof of (1) implies
(4) in Theorem \ref{thm:Main-intro} using stationary logic}

\subsection{More on clubs}
\begin{defn}
Suppose $\kappa$ is a cardinal and $A\subseteq B$, $S\subseteq\left[A\right]^{\kappa}$,
and $T\subseteq\left[B\right]^{\kappa}$. We define $S^{B}\in\left[A\right]^{\kappa}$
and $T\upharpoonright A\in\left[A\right]^{\kappa}$ by 
\begin{eqnarray*}
S^{B} & = & \set{Y\in\left[B\right]^{\kappa}}{Y\cap A\in S}\\
T\upharpoonright A & = & \set{X\in\left[A\right]^{\kappa}}{\text{ there is }Y\in T\text{ such that }X=Y\cap A}.
\end{eqnarray*}
\end{defn}

\begin{fact}
\cite[Theorem 8.27]{jech2013set} \label{jechfacts} Suppose $\kappa$
is a cardinal, $A\subseteq B$, $S\subseteq\left[A\right]^{\kappa}$,
and $T\subseteq\left[B\right]^{\kappa}$. 
\begin{enumerate}
\item If $S$ is stationary in $\left[A\right]^{\kappa}$, then $S^{B}$
is stationary in $\left[B\right]^{\kappa}$. 
\item If $T$ is stationary in $\left[B\right]^{\kappa}$, then $T\upharpoonright A$
is stationary in $\left[A\right]^{\kappa}$. 
\end{enumerate}
\end{fact}

\begin{lem}
\label{stationaryunion} Suppose $X$ is a set and $\lambda$ and
$\kappa$ are cardinals with $\lambda\leq\kappa<\left|X\right|$.
Suppose, moreover, we are given a stationary subset $\mathcal{S}\subseteq\left[X\right]^{\kappa}$
and, for every $Y\in\mathcal{S}$, a stationary subset $S_{Y}\in\left[Y\right]^{\lambda}$.
Then $\mathcal{S}'=\bigcup_{Y\in\mathcal{S}}S_{Y}$ is a stationary
subset of $\left[X\right]^{\lambda}$. 
\end{lem}

\begin{proof}
Suppose $D\subseteq\left[X\right]^{\lambda}$ is a club. We must show
$\mathcal{S}'\cap D\neq\emptyset$. By Fact \ref{clubfacts}(3), there
is a sequence of finitary functions $\overline{f}=\sequence{f_{i}}{i<\lambda}$
where for all $i<\lambda$, $f_{i}:X^{n_{i}}\to X$ and the set $C_{\overline{f}}\subseteq\left[X\right]^{\lambda}$
of $\lambda$-sized subsets of $X$ closed under $\overline{f}$ is
a club with $C_{\overline{f}}\subseteq D$. The subsets of $X$ of
size $\kappa$ closed under $\overline{f}$ form a club $C_{\overline{f}}^{*}\subseteq\left[X\right]^{\kappa}$,
hence $C_{\overline{f}}^{*}\cap\mathcal{S}\neq\emptyset$. Fix $Y\in C_{\overline{f}}^{*}\cap\mathcal{S}$.
Define a sequence of functions $\overline{g}=\sequence{g_{i}}{i<\lambda}$
by $g_{i}=f_{i}\upharpoonright Y^{n_{i}}$ for all $i<\lambda$. This
definition makes sense as $Y$ is closed under the functions $f_{i}$
so that $C_{\overline{f}}\cap\left[Y\right]^{\lambda}=C_{\overline{g}}$,
the subsets of $Y$ closed under $\overline{g}$, hence is a club
of $\left[Y\right]^{\lambda}$. Therefore $C_{\overline{f}}\cap\left[Y\right]^{\lambda}\cap S_{Y}\neq\emptyset$.
In particular, this shows $D\cap S'\neq\emptyset$, which completes
the proof. 
\end{proof}
The club filter on $\left[X\right]^{\omega}$ was characterized by
Kueker in terms of games of length $\omega$ \cite{kueker1972lowenheim}.
The natural analogue for games of length $\lambda$ determines a filter
on $\mathcal{P}_{\lambda^{+}}\left(X\right)$, which, in general,
differs from the club filter. In generalizing stationary logic to
quantification over sets of some uncountable size $\lambda$, it turns
out that this filter provides a more useful analogue to the club filter
on $\left[X\right]^{\omega}$ than the club filter on $\left[X\right]^{\lambda}$.
\begin{defn}
Suppose $X$ is a set and $\lambda$ is a regular cardinal. Given
a subset $F\subseteq\mathcal{P}_{\lambda^{+}}\left(X\right)$, we
define the game $G\left(F\right)$, to be the game of length $\lambda$
where Players I and II alternate playing an increasing $\lambda$
sequence of elements of $\mathcal{P}_{\lambda^{+}}\left(X\right)$.
In this game, Player II wins if and only if the union of the sets
played is in $F$. The filter $D_{\lambda}\left(X\right)$ is defined
to be the filter generated by the sets $F\subseteq\mathcal{P}_{\lambda^{+}}\left(X\right)$
in which Player II has a winning strategy in $G\left(F\right)$. We
say $Y\subseteq\mathcal{P}_{\lambda^{+}}\left(X\right)$ is $D_{\lambda}\left(X\right)$\emph{-stationary}
if $Y$ intersects every set in $D_{\lambda}\left(X\right)$.
\end{defn}

It is easy to check that every club $C\subseteq\left[X\right]^{\lambda}$
is an element of $D_{\lambda}(X)$ and, therefore, that every $S\subseteq\left[X\right]^{\lambda}$
that is $D_{\lambda}\left(X\right)$-stationary is also stationary
with respect to the usual club filter on $\left[X\right]^{\lambda}$.
It was remarked in \cite{mekler1986stationary} that if $\lambda=\lambda^{<\lambda}$,
then $D_{\lambda}\left(\lambda^{+}\right)$ is just the filter generated
by the clubs of $\lambda^{+}$ intersected with the set of ordinals
of cofinality $\lambda$ (considered as initial segments of $\lambda^{+}$).
More precisely, we have the following fact. (We omit its proof since
it is not necessary for the rest.)
\begin{fact}
\label{fancyclubfacts} Suppose $\lambda$ is an infinite cardinal
and write $S_{\lambda}^{\lambda^{+}}$ for the stationary set $\set{\alpha<\lambda^{+}}{\cof\left(\alpha\right)=\lambda}$. 
\begin{enumerate}
\item If $C\subseteq\lambda^{+}$ is a club, then $C\cap S_{\lambda}^{\lambda^{+}}\in D_{\lambda}\left(\lambda^{+}\right)$. 
\item Suppose $\lambda=\lambda^{<\lambda}$. Then $D_{\lambda}\left(\lambda^{+}\right)$
is generated by sets of the form $C\cap S_{\lambda}^{\lambda^{+}}$,
where $C\subseteq\lambda^{+}$ is a club.
\end{enumerate}
\end{fact}

\begin{lem}
\label{realstationary} Suppose $X$ is a set of size $\lambda^{+}$,
and $\sequence{X_{\alpha}}{\alpha<\lambda^{+}}$ is an increasing
and continuous sequence from $\mathcal{P}_{\lambda^{+}}\left(X\right)$
with union $X$. Suppose $S\subseteq\mathcal{P}_{\lambda^{+}}\left(X\right)$
is $D_{\lambda}\left(X\right)$-stationary. Then the set $S_{*}=\set{\alpha<\lambda^{+}}{\cof(\alpha)=\lambda,X_{\alpha}\in S}$
is a stationary subset of $\lambda^{+}$.
\end{lem}

\begin{proof}
As $\left|X\right|=\lambda^{+}$, we may assume $X=\lambda^{+}$.
Let $C\subseteq\lambda^{+}$ consist of the ordinals $\alpha<\lambda^{+}$
such that $X_{\alpha}=\alpha$. This set is easily seen to be a club.

Let $C_{*}\subseteq\lambda^{+}$ be a club. We must show $C_{*}\cap S_{*}\neq\emptyset$.
By Fact \ref{fancyclubfacts}(1), $C_{*}\cap C\cap S_{\lambda}^{\lambda^{+}}\in D_{\lambda}\left(X\right)$,
hence $\left(S_{\lambda}^{\lambda^{+}}\cap C\cap C_{*}\right)\cap S\neq\emptyset$.
Pick $Y$ in this intersection. Then by definition of $C$, $Y=X_{\alpha}=\alpha$
for some $\alpha\in S_{\lambda}^{\lambda^{+}}$. As $X_{\alpha}\in S$,
we have $\alpha\in S_{*}$. This shows $S_{*}\cap C_{*}\neq\emptyset$.
\end{proof}
\begin{lem}
\label{fancyrestriction} Suppose $A\subseteq B$ and $S\subseteq\mathcal{P}_{\lambda^{+}}\left(B\right)$
is $D_{\lambda}\left(B\right)$-stationary. Then the set $S\upharpoonright A=\set{X\cap A}{X\in S}$
is $D_{\lambda}\left(A\right)$-stationary.
\end{lem}

\begin{proof}
It is enough to show that if $F\in D_{\lambda}\left(A\right)$ then
$F^{B}=\set{X\in\mathcal{P}_{\lambda^{+}}\left(B\right)}{X\cap A\in F}\in D_{\lambda}\left(B\right)$.
We may assume that there is some winning strategy $f$ for Player
II in the game $G\left(F\right)$, since $F\in D_{\lambda}\left(A\right)$.
That is, the function $f$ is defined so that if, at stage $i$, Player
I has played $\sequence{A_{j}}{j\leq i}$ then $f\left(\sequence{A_{j}}{j\leq i}\right)$
outputs the play for Player II.

Now we will define a winning strategy for Player II in the game $G\left(F^{B}\right)$.
At stage $i$, if Player I has played $\sequence{A_{j}}{j\leq i}$,
Player II plays $B_{i}=A_{i}\cup f\left(\sequence{A_{j}\cap A}{j\leq i}\right)$.
As the rules of the game require that the sets are increasing, we
have 
\[
A_{i}\cap A\subseteq f\left(\sequence{A_{j}\cap A}{j\leq i}\right)\subseteq A,
\]
hence $B_{i}\cap A=f\left(\sequence{A_{j}\cap A}{j\leq i}\right)$.
It follows that

\begin{center}%
\begin{tabular}{l|ccccc}
I  & $A_{0}\cap A$  &  & $A_{1}\cap A$  &  & $\cdots$ \tabularnewline
\hline 
II  &  & $B_{0}\cap A$  &  & $B_{1}\cap A$  & $\cdots$ \tabularnewline
\end{tabular}\end{center} is a play according to $f$ in $G\left(F\right)$. Therefore,
\[
\left(\bigcup_{i<\lambda}A_{i}\cup B_{i}\right)\cap A=\bigcup_{i<\lambda}\left(A_{i}\cap A\right)\cup\left(B_{i}\cap A\right)\in F,
\]
which shows $\bigcup_{i<\lambda}A_{i}\cup B_{i}\in F^{B}$. We have
shown that Player II has a winning strategy in $G\left(F^{B}\right)$
so $F^{B}\in D_{\lambda}\left(B\right)$.
\end{proof}

\subsection{Stationary logic}

The stationary logic $L\left(\textsf{aa}\right)$ was introduced in
\cite{shelah1975generalized} (where it was called $L\left(Q_{\aleph_{1}}^{ss}\right)$).
The logic is defined as follows: given a first-order language $L$,
expand the language with countably many new unary predicates $\set{S_{i}}{i<\omega}$
and a new quantifier $\textsf{aa}$. The formulas of $L$ in $L\left(\textsf{aa}\right)$
are the the smallest class containing the first-order formulas of
$L$, closed under the usual first-order formation rules together
with the rule that if $\varphi$ is a formula, then $\left(\textsf{aa}S_{i}\right)\varphi$
is also a formula, for any new unary predicate $S_{i}$. Satisfaction
is defined as usual, together with the rule that $M\models\left(\textsf{aa}S\right)\varphi\left(S\right)$
if and only if $M\models\varphi\left(S\right)$ when $S^{M}=X$ for
``almost all'' $X\in\left[M\right]^{\omega}$\textemdash that is,
$\set{X\in\left[M\right]^{\omega}}{\text{ if }S^{M}=X\text{ then }M\models\varphi\left(S\right)}$
contains a club of $[M]^{\omega}$. We define the quantifier $\textsf{stat}$
dually: $M\models\left(\textsf{stat}S\right)\varphi\left(S\right)$
if and only if $M\models\neg\left(\textsf{aa}S\right)\neg\varphi\left(S\right)$.
Note that $M\models\left(\textsf{stat}S\right)\varphi\left(S\right)$
if and only if $\set{X\in\left[M\right]^{\omega}}{\text{ if }S^{M}=X\text{ then }M\models\varphi\left(S\right)}$
is stationary. Given an $L$-structure $M$, we write $\text{Th}_{\textsf{aa}}\left(M\right)$
for the set of $L\left(\text{aa}\right)$-sentences satisfied by $M$.
We refer the reader to \cite[Section 1]{barwise1978stationary} for
a detailed treatment of stationary logic.

Later work by Mekler and the third-named author extended stationary
logic, which quantifies over \emph{countable sets}, to a logic that
permits quantification over sets of higher cardinality \cite{mekler1986stationary}.
For $\lambda$ a regular cardinal, the logic $L\left(\aa^{\lambda}\right)$
is defined analogously to $L\left(\aa\right)$, with semantics defined
so that $M\models\left(\mathsf{aa}^{\lambda}S\right)\varphi\left(S\right)$
if and only if $\set{X\in\left[M\right]^{\lambda}}{\text{ if }S^{M}=X\text{ then }M\models\varphi(S)}\in D_{\lambda}\left(M\right)$.
The quantifier $\mathsf{stat}^{\lambda}$ is also understood dually:
$M\models\left(\mathsf{stat}^{\lambda}S\right)\varphi\left(S\right)$
if and only if $M\models\neg\left(\mathsf{aa}^{\lambda}S\right)\neg\varphi\left(S\right)$.
If $T$ is an $L\left(\mathsf{aa}\right)$-theory, one obtains an
$L\left(\mathsf{aa}^{\lambda}\right)$-theory by replacing the quantifier
$\mathsf{aa}$ with $\mathsf{aa}^{\lambda}$. We call this theory
the $\lambda$\emph{-interpretation of} $T$. By working with $D_{\lambda}\left(M\right)$
instead of the full club filter on $\left[M\right]^{\lambda}$, one
is able to relate satisfiability of an $L\left(\mathsf{aa}\right)$-theory
to the satisfiability of its $\lambda$-interpretation. Below, the
``moreover'' clause about $\lambda$-saturation is not stated in
\cite{mekler1986stationary}, but is immediate from the proof.
\begin{fact}
\label{lambdainterp} \cite[Theorem 1.3]{mekler1986stationary} Suppose
$\lambda=\lambda^{<\lambda}$ and $T$ is a consistent $L\left(\aa\right)$-theory
of size at most $\lambda$. Then the $\lambda$-interpretation of
$T$ has a model of size at most $\lambda^{+}$. In fact, there is
such a model which is, moreover, $\lambda$-saturated.
\end{fact}

The following easy observation is also useful:
\begin{lem}
\label{nottoosmall} Suppose $\varphi$ is a first-order formula,
possibly with parameters from $M$ and $\left|\varphi\left(M\right)\right|>\aleph_{0}$.
Then if $M'\models\text{Th}_{\textsf{aa}}\left(M\right)$ in the $\lambda$-interpretation,
then $\left|\varphi\left(M'\right)\right|>\lambda$.
\end{lem}

\begin{proof}
Suppose not. Then $\set{S\in\left[M'\right]^{\lambda}}{\varphi\left(M'\right)\subseteq S}$
is a club of $\left[M'\right]^{\lambda}$ hence an element of $D_{\lambda}\left(M'\right)$.
Therefore $M'\models\left(\mathsf{aa}^{\lambda}S\right)\forall x\left(\varphi\left(x\right)\to S\left(x\right)\right)$.
As $M'\models\text{Th}_{\textsf{aa}}\left(M'\right)$ in the $\lambda$-interpretation,
$M\models\left(\mathsf{aa}S\right)\forall x\left(\varphi\left(x\right)\to S\left(x\right)\right)$,
so $\varphi\left(M\right)$ is countable, a contradiction.
\end{proof}

\subsection{Reduction to a countable language}
\begin{rem}
\label{rem:changing language doesn't matter for K-dividing}Suppose
that $T$ is an NSOP$_{1}$ theory in the language $L$. Suppose that
$M\models T$ and $\varphi\left(x,y\right)$ is any formula. Then
for any language $L'\subseteq L$ containing $\varphi$, and any $b\in\Mm$,
$\varphi\left(x,b\right)$ Kim-divides over $M$ in $L$ iff $\varphi\left(x,b\right)$
Kim-divides over $M':=M\restriction L'$ (in the sense of $T\restriction L'$).
Indeed, this follows from Kim's lemma for Kim-dividing (Fact \ref{kimslemma})
and the fact that if $\bar{b}$ is a coheir sequence in $L$ over
$M$ starting with $b$, then it is also in $L'$. 
\end{rem}

\begin{lem}
\label{reducetocountable} Suppose $T$ is an NSOP$_{1}$ theory in
the language $L$, $M\models T$ and for some $p\in S\left(M\right)$,
the set 
\[
S=\set{N\prec M}{\left|N\right|=\left|T\right|,\text{ }p\text{ Kim-divides over }N}
\]
is stationary in $\left[M\right]^{\left|T\right|}$. Then there is
a countable sublanguage $L'\subseteq L$ and a stationary set $S'\subseteq\left[M\right]^{\omega}$
so that, setting $p'=p\upharpoonright L$, we have that for all $N'\in S'$,
$p'$ Kim-divides over $N'$.
\end{lem}

\begin{proof}
 For each $N\in S$, choose some $\varphi_{N}\left(x;b_{N}\right)\in p$
such that $\varphi\left(x;b_{N}\right)$ Kim-divides over $N$. By
Fact \ref{clubfacts}(1), the club filter on $\left[M\right]^{\left|T\right|}$
is $|T|^{+}$-complete, so for any a partition of a stationary set
into $\left|T\right|$ many pieces, we may find some piece which is
stationary. Therefore we may assume there is some formula $\varphi$
so that $\varphi_{N}\left(x;b_{N}\right)=\varphi\left(x;b_{N}\right)$
for all $N\in S$. Let $L'$ be any countable sublanguage of $L$
containing $\varphi$. By Remark \ref{rem:changing language doesn't matter for K-dividing}
and (the proof of) Theorem \ref{clubinbase}, for each $N\in S$,
there is a club $C_{N}\subseteq\left[N\right]^{\omega}$ of countable
$L'$-elementary substructures over which $\varphi\left(x;b_{N}\right)$
Kim-divides. By Lemma \ref{stationaryunion}, $S'=\bigcup_{N\in S}C_{N}$
is a stationary subset of $\left[M\right]^{\omega}$. By definition
of $S'$, if $N'\in S'$, then there is some $\varphi\left(x;b_{N'}\right)\in p$
such that $\varphi\left(x;b_{N'}\right)$ Kim-divides over $N'$. 
\end{proof}

\subsection{Stretching}
\begin{lem}
\label{aleph1} Suppose $T$ is NSOP$_{1}$, $\left|T\right|=\aleph_{0}$,
$M\models T$, and there is $p\in S\left(M\right)$ so that the set
\[
S_{0}=\set{N\prec M}{\left|N\right|=\aleph_{0}\text{ and }p\text{ Kim-divides over }N}
\]
is stationary. Then, given any regular uncountable cardinal $\lambda=\lambda^{<\lambda}$,
there is a model $M'\models T$, $\left|M'\right|=\lambda^{+}$, a
formula $\varphi\left(x;y\right)$, and a type $p_{*}$ over $M'$
so that 
\[
S_{0}'=\set{N'\prec M'}{\left|N'\right|=\lambda,\text{ there is }\varphi\left(x;a'_{N}\right)\in p_{*}\text{ that Kim-divides over }N'}
\]
is $D_{\lambda}\left(M'\right)$-stationary.
\end{lem}

\begin{proof}
As no type Kim-divides over its domain, it follows that $M$ is uncountable.
For each $N\in S_{0}$, there is some formula $\varphi_{N}\left(x;a_{N}\right)\in p$
and $k_{N}<\omega$ so that $\varphi\left(x;a_{N}\right)$ $k_{N}$-Kim-divides
over $N$ via a Morley sequence in some global $N$-finitely satisfiable
type. As the club filter on $\left[M\right]^{\omega}$ is $\aleph_{1}$-complete,
Fact \ref{clubfacts}(1), there are $\varphi$ and $k$ so that for
some stationary $S\subseteq S_{0}$, we have $N\in S$ implies $\varphi_{N}\left(x;a_{N}\right)=\varphi\left(x;a_{N}\right)$
and $k_{N}=k$.

Let $l=\left|a_{N}\right|$ for all $N\in S$ and let $\tilde{M}$
be an $\aleph_{1}$-saturated elementary extension of $M$. Let $\chi$
be a sufficiently large regular cardinal so that all objects of interest
are contained in $H\left(\chi\right)$. In particular, we may choose
$\chi$ so that $\tilde{M},{^{\omega}\tilde{M}},T$, $L$, and $p$
are contained in $H\left(\chi\right)$, together with a bijection
to $\omega$ witnessing the countable cardinality of $L$, and we
consider the structure 
\[
\mathcal{H}=\left(H\left(\chi\right),\in,M,\tilde{M},L,T,p\right).
\]
By Fact \ref{jechfacts}(1), the set $S_{*}=\set{X\in\mathcal{\left[H\right]}^{\omega}}{X\cap M\in S}$
is a stationary subset of $\mathcal{\left[H\right]}^{\omega}$.

Let $\Phi\left(X\right)$ be the formula in the language of $\mathcal{H}$
together with a new predicate $X$ that naturally asserts: there exists
$c\in M^{l}$, such that $\varphi\left(x;c\right)\in p$ and such
that there exists $f\in{^{\omega}\left(\tilde{M}^{l}\right))}$ such
that: 

$X\cap M$ is an elementary substructure of $M$. 

$f=\sequence{f_{i}}{i<\omega}$ is an $\left(X\cap M\right)$-indiscernible
sequence such that $\text{tp}\left(f_{i}/\left(X\cap M\right)f_{<i}\right)$
is finitely satisfiable in $\left(X\cap M\right)$. 

$f\left(0\right)=c$. 

$\set{\varphi\left(x;f_{i}\right)}{i<\omega}$ is $k$-inconsistent. 

We first show the following:
\begin{claim*}
$\mathcal{H}\models\left(\textsf{stat}X\right)\Phi\left(X\right)$.
\end{claim*}
\begin{proof}[\emph{Proof of claim}]
As $S_{*}$ is stationary, it suffices to show that if $X\in S_{*}$
and $S^{\mathcal{H}}=S_{*}$ then $\mathcal{H}\models\Phi\left(S\right)$.
Recall that if $X\in S_{*}$, then $X\cap M\in S$ so $X\cap M$ is
a countable elementary substructure of $M$, and $\varphi\left(x;a_{X\cap M}\right)$
is a formula in $p$ that $k$-Kim-divides over $X\cap M$. As $\tilde{M}$
is $\aleph_{1}$-saturated, there is a coheir sequence $\sequence{a_{i}}{i<\omega}$
over $X\cap M$ in $\tilde{M}$ with $a_{0}=a_{X\cap M}$ and $\set{\varphi\left(x;a_{i}\right)}{i<\omega}$
$k$-inconsistent. Put $c=a_{0}$ and let $f\in{^{\omega}(\tilde{M}^{l})}$
be defined by $f_{i}=a_{i}$, we easily have (1)-(4) satisfied, proving
the claim. 
\end{proof}
By Fact \ref{lambdainterp}, there is $\mathcal{H}'$ which is a model
of the $\lambda$-interpretation of $\text{Th}_{aa}\left(\mathcal{H}\right)$
with $\left|\mathcal{H}'\right|=\lambda^{+}$, $\mathcal{H}'=\left(\mathcal{H}',\in',M',\tilde{M}',L',p'\right)$.
As $L$ and $T$ are coded by natural numbers, the language $L$ is
contained in $L'$ and thus the definable set $\set{x\in\mathcal{H}'}{x\in\tilde{M}'}$
may be regarded as the domain of an $L'$-structure whose reduct to
$L$ is a model of $T$ and likewise for $M$. Moreover $M'\prec_{L}\tilde{M}'$
and $\left|M'\right|=\lambda^{+}$, by Lemma \ref{nottoosmall}. As
$\mathcal{H}'\models\left(\textsf{stat}^{\lambda}X\right)\Phi\left(X\right)$,
there is a $D_{\lambda}\left(\mathcal{H}'\right)$-stationary set
$S_{*}'\subseteq\left[\mathcal{H}'\right]^{\lambda}$ witnessing this.
Let $S'=S_{*}'\upharpoonright M'$\textemdash i.e. $S'=\set{X\cap M'}{X\in S_{*}'}$.
By Lemma \ref{fancyrestriction}, $S'$ is $D_{\lambda}\left(M'\right)$-stationary.
Let $p_{*}=p'\upharpoonright L$. To conclude the proof, it suffices
to establish the following:
\begin{claim*}
$p_{*}$ is a type over $M'$ and if $N\in S'$, then $p_{*}$ $k$-Kim-divides
over $N$ via some $\varphi\left(x;a_{N}'\right)\in p_{*}$.
\end{claim*}
\begin{proof}[\emph{Proof of claim}]
It is clear that $p_{*}$ is a consistent type over $M'$. Now fix
$N\in S'$. By definition of $S'$, $N=X\cap M'$ for some $X\in\left[\mathcal{H}'\right]^{\lambda}$
such that $\mathcal{H}'\models\Phi\left(S\right)$ when $S^{\mathcal{H}'}=X$.
It follows that for some $b\in M'$, there is an $N$-indiscernible
sequence $\sequence{b_{i}}{i\in I}$ with $b_{0}=b$, such that $\text{tp}\left(b_{i}/Nb_{<i}\right)$
is finitely satisfiable in $N$, $\varphi\left(x;b_{0}\right)\in p'$
and $\set{\varphi\left(x;b_{i}\right)}{i\in I}$ is $k$-inconsistent,
where $I$ denotes the (possibly non-standard) natural numbers of
$\mathcal{H}'$. By indiscernibility, $\sequence{b_{i}}{i\in I}$
is a Morley sequence over $N$ in a global $N$-finitely satisfiable
type, which shows $\varphi\left(x;b_{0}\right)$ $k$-Kim-divides
over $N$. This completes the proof. 
\end{proof}
\renewcommand{\qedsymbol}{}
\end{proof}

\subsection{The main lemma}
\begin{lem}
\label{hard direction}(Main Lemma) Suppose $T$ is a complete theory,
$M\models T$ is a model with $\left|M\right|\geq\left|T\right|$,
and for some $p\in S\left(M\right)$, the set 
\[
S_{0}=\set{N\in\left[M\right]^{\left|T\right|}}{N\prec M,p\text{ Kim-divides over }N}
\]
is stationary. Then $T$ has SOP$_{1}$.
\end{lem}

\begin{proof}
Towards contradiction suppose $T$ is NSOP$_{1}$. By Lemma \ref{reducetocountable},
there is a countable sublanguage $L'\subseteq L$ and a stationary
set $S_{0}'\subseteq\left[M\right]^{\omega}$ such that if $p'=p\upharpoonright L'$
then for all $N\in S$, $N\prec_{L'}M$ and $p'$ Kim-divides over
$N$. Therefore, we may assume for the rest of the proof that $T$
is countable.

By forcing with the Lévy collapse $\text{Coll}\left(\lambda^{+},2^{\lambda}\right)$
for a sufficiently large cardinal, we may assume there is some an
uncountable cardinal $\kappa=\kappa^{<\kappa}$, namely $\kappa=\lambda^{+}$,
while preserving the situation. By Lemma \ref{aleph1}, there is a
model $M'\models T$ with $\left|M'\right|=\kappa^{+}$ and a type
$p'$ over $M'$ so that 
\[
S''_{0}=\set{N\in\left[M'\right]^{\kappa}}{N\prec M'\text{ and some }\varphi\left(x;c_{N}\right)\in p'\text{ Kim-divides over }N}
\]
is $D_{\kappa}\left(M'\right)$-stationary. Let $\sequence{M_{\alpha}}{\alpha<\kappa^{+}}$
be a continuous and increasing sequence of $\kappa$-sized elementary
substructures of $M'$ with union $M'$. The set $S=\set{\alpha<\kappa^{+}}{\text{cf}\left(\alpha\right)=\kappa,M_{\alpha}\in S_{0}''}$
is a stationary subset of $\kappa^{+}$ by Lemma \ref{realstationary}.
By intersecting with a club, we may also assume that for all $\alpha\in S$,
$M_{\alpha}$ contains $c_{M_{\beta}}$ for all $\beta\in\alpha\cap S$.

From here, the proof closely follows the proof of \cite[Theorem 4.5]{kaplan2017kim}.
For each $\alpha\in S$, let $c_{\alpha}$ denote $c_{M_{\alpha}}$
and let $r_{\alpha}$ be a global $M_{\alpha}$-finitely satisfiable
type extending $\tp\left(c_{\alpha}/M_{\alpha}\right)$. By reducing
$S$, we may assume that there is some $k<\omega$ such that $r$
witnesses that $\varphi\left(x;c_{\alpha}\right)$ $k$-Kim-divides
over $M_{\alpha}$. For each $\alpha\in S$, apply Lemma \ref{clubfinding}(2)
to choose a countable $N_{\alpha}\prec M_{\alpha}$ such that $r_{\alpha}^{\otimes\omega}|_{N_{\alpha}}$
is the type of a Morley sequence in some global $N_{\alpha}$-finitely
satisfiable type and hence such that $\varphi\left(x;c_{\alpha}\right)$
$k$-Kim-divides over $N_{\alpha}$. Define $\rho:S\to\kappa^{+}$
by $\rho\left(\alpha\right)=\min\set{\beta<\alpha}{N_{\alpha}\subseteq M_{\beta}}$.
This is well-defined and pressing down on $S$ as $\kappa$ is regular
and uncountable. By Fodor's lemma, there is $S'\subseteq S$ such
that $\rho$ is constant on $S'$, say with constant value $\beta_{0}$.
As $\left|M_{\beta_{0}}\right|=\kappa$, there are $\leq\kappa^{\aleph_{0}}=\kappa$
many choices for $N_{\alpha}\subseteq M_{\beta_{0}}$ so there is
a stationary $S''\subseteq S'$ and $N'_{0}$ so that $N_{\alpha}=N'_{0}$
for all $\alpha\in S''$. As there are $\leq2^{\aleph_{0}}\leq\kappa$
choices for $r_{\alpha}^{\otimes\omega}|_{N'_{0}}$, there is a stationary
$S_{*}\subseteq S''$ such that $r_{\alpha}^{\otimes\omega}|_{N'_{0}}$
is constant, with value $s_{0}^{\otimes\omega}|_{N'_{0}}$ for some
global coheir $s_{0}$ over $N'_{0}$. Let $\delta_{0}=\min S_{0}$,
$e_{0}=c_{\delta_{0}}$.

Repeating this process $\omega$ many times, we find an increasing
sequence $\sequence{\delta_{i}}{i<\omega}$ of ordinals in $\kappa^{+}$,
an increasing sequence of models $\sequence{N_{i}'}{i<\omega}$, $e_{i}\in M'$
for $i<\omega$ and global $N_{i}'$-finitely satisfiable types $s_{i}$
such that: 

$N_{i}'$ contains $e_{<i}$, $\varphi\left(x;e_{j}\right)$ is $k$-Kim-dividing
over $N_{i}'$ for every $i\leq j$, $s_{i}$ is a global coheir over
$N'_{i}$ extending $\text{tp}\left(e_{i}/N_{i}'\right)$ and $e_{j}\equiv_{N'_{i}}e_{i}$
for all $j\geq i$. In addition, $s_{j}^{\otimes\omega}|_{N'_{i}}=s_{i}^{\otimes\omega}|_{N_{i}'}$
for all $j\geq i$. 

Denote $\overline{e}=\langle e_{i}:i<\omega\rangle$. Note that $\set{\varphi\left(x;e_{i}\right)}{i<\omega}$
is a subset of $p'$, hence consistent. Now, exactly as in the claim
in the proof of \cite[Theorem 4.5]{kaplan2017kim}, we can show that
if $i_{0}<\ldots<i_{n-1}<\omega$ and for each $j<n$, $f_{j}\models s_{i_{j}}|_{N'_{i_{j}}\overline{e}f_{>j}}$
then $e_{i_{j}}\equiv_{e_{i_{<j}}f_{<j}}f_{j}$ for all $j<n$ and
$\set{\varphi\left(x;f_{j}\right)}{j<n}$ is $k$-inconsistent. By
compactness, we can find an array $\sequence{\left(c_{i,0},c_{i,1}\right)}{i<\omega}$
so that $\set{\varphi\left(x,c_{i,0}\right)}{i<\omega}$ is consistent,
$\set{\varphi\left(x,c_{i,1}\right)}{i<\omega}$ is $k$-inconsistent,
and $c_{i,0}\equiv_{\overline{c}_{<i}}c_{i,1}$ for all $i<\omega$.
By Fact \ref{karyversion}, we obtain SOP$_{1}$, a contradiction.
\end{proof}
\begin{cor}
\label{cor:(1) implies (4)} Theorem \ref{thm:Main-intro} (1) $\implies$
(4) holds.
\end{cor}

\section{\label{sec:Dual-local-character}Dual local character}
\begin{defn}
\label{def:strong Kim-dividing} ($T$ any theory) Say that a formula
$\varphi\left(x,a\right)$ \emph{strongly Kim-divides} over a model
$M$ if for every global $M$-invariant type $q\supseteq\tp\left(a/M\right)$,
$\varphi\left(x,a\right)$ Kim-divides\emph{ over }$M$\emph{ via
}$q$.
\end{defn}

\begin{rem}
\label{rem:strong Kim-dividing =00003D Kim div in NSOP1} By Fact
\ref{kimslemma}, strong Kim-dividing = Kim-dividing iff $T$ is NSOP$_{1}$. 
\end{rem}

\begin{defn}
A \emph{dual type} (over $A$) in $x$ is a set $F$ of ($A$-)definable
sets in $x$ such that for some $k<\omega$, it is $k$-inconsistent.
Say that $F$ \emph{dually divides} over a model $N$, if every $X\in F$
which is not definable over $N$ divides over $N$. Similarly define
when $F$ \emph{dually Kim-divides} over $N$ and when $F$ \emph{strongly
dually Kim-divides} over $N$. 
\end{defn}

\begin{thm}
\label{thm:dual local character}The following are equivalent for
a complete theory $T$.
\begin{enumerate}
\item $T$ is NSOP$_{1}$. 
\item There is no continuous increasing sequence of $|T|$-sized models
$\sequence{M_{i}}{i<\left|T\right|^{+}}$ with union $M$ and a dual
type $F$ over $M$ such that $F\restriction M_{i+1}$ does not strongly
dually Kim-divide over $M_{i}$ for all $i<|T|^{+}$. 
\item Assume that $M\models T$ and $F$ a dual type over $M$. Then there
is a stationary subset $S$ of $\left[M\right]^{\left|T\right|}$
such that if $N\in S$ then $N\prec M$ and $F$ strongly dually Kim-divides
over $N$. 
\item (Dual local character) Same as (3) but $S$ is a club. 
\end{enumerate}
\end{thm}

\begin{proof}
The proof is essentially dualizing or inverting the proof (using stationary
logic) of Theorem \ref{thm:Main-intro} (1) $\implies$ (4), but we
go into some details. 

(1) $\implies$ (4). We follow the proof of ``(1) implies (4)''
of Theorem \ref{thm:Main-intro} as described in Section \ref{sec:A proof using stationary logic}.
Namely, assume that (2) fails. This means that there is a stationary
subset $S$ of $\left[M\right]^{\left|T\right|}$ such that if $N\in S$
then $N\prec M$ and there is some $X\in F$ which is not definable
over $N$ but still does not Kim-divide over $N$. Using the same
proof as in Lemma \ref{reducetocountable}, we may assume that the
language $L$ is countable and that there is a single formula $\varphi\left(x,y\right)$
with $\left|x\right|=n$ such that if $N\in S$ then for some $b\in M\backslash N$,
$\varphi\left(x,b\right)$ does not Kim-divide over $N$ (and $\varphi\left(x,b\right)$
is not $N$-definable). Now we repeat the same procedure as in Lemma
\ref{aleph1}. Thus, for a regular uncountable cardinal $\lambda=\lambda^{<\lambda}$,
we get a model $M'\models T$, $\left|M'\right|=\lambda^{+}$, a formula
$\varphi\left(x,y\right)$, and a $k$-inconsistent family $F_{*}$
of definable formulas over $M'$ so that
\[
S_{0}'=\set{N'\prec M'}{\left|N'\right|=\lambda,\exists\varphi\left(x;a'_{N}\right)\in F_{*}\text{ not \ensuremath{N'}-definable and does not Kim-divide over }N'}
\]
is $D_{\lambda}\left(M'\right)$-stationary. Now we repeat the proof
of Lemma \ref{hard direction}. The contradiction we will arrive at
the end will be the same contradiction, but the roles of the sequences
$e_{i}$ and $f_{j}$ are reversed. Now $\set{\varphi\left(x,e_{i}\right)}{i<\omega}$
is $k$-inconsistent (note that the formulas $\varphi\left(x,e_{i}\right)$
must define distinct definable sets from $F_{*}$) and $\sequence{\varphi\left(x,f_{j}\right)}{j<n}$
is consistent. 

(4) $\implies$ (3) $\implies$ (2) is exactly as in the proof of
Lemma \ref{mainthm(1)-(6)}. The proof of (2) $\implies$ (1) is just
dualizing the proof of ``(2) implies (1)'' in Theorem \ref{mainthm(1)-(6)}
in the sense that the sequences $\sequence{c_{i,0}}{i<\omega}$ and
$\sequence{c_{i,1}}{i<\omega}$ exchange places.
\end{proof}
\begin{question}
Is there a proof of the dual local character which does not use stationary
logic? Such a proof may reveal some new properties of Kim-dividing. 
\end{question}

\bibliographystyle{alpha}
\bibliography{ms}

\begin{thebibliography}{BKM78}

\bibitem[BKM78]{barwise1978stationary}
Jon Barwise, Matt Kaufmann, and Michael Makkai.
\newblock Stationary logic.
\newblock {\em Annals of Mathematical Logic}, 13(2):171--224, 1978.

\bibitem[DS04]{dvzamonja2004maximality}
Mirna D{\v{z}}amonja and Saharon Shelah.
\newblock On $\vartriangleleft^{*}$-maximality.
\newblock {\em Annals of Pure and Applied Logic}, 125(1):119--158, 2004.

\bibitem[Jec13]{jech2013set}
Thomas Jech.
\newblock {\em Set theory}.
\newblock Springer Science \&amp; Business Media, 2013.

\bibitem[KR17]{kaplan2017kim}
Itay Kaplan and Nicholas Ramsey.
\newblock On kim-independence.
\newblock {\em J. Eur. Math. Soc. (JEMS)}, 2017.
\newblock accepted, \url{arXiv:1702.03894}.

\bibitem[Kue72]{kueker1972lowenheim}
David~W Kueker.
\newblock L{\"o}wenheim-skolem and interpolation theorems in infinitary
  languages.
\newblock {\em Bulletin of the American Mathematical Society}, 78(2):211--215,
  1972.

\bibitem[MS86]{mekler1986stationary}
Alan~H Mekler and Saharon Shelah.
\newblock Stationary logic and its friends. ii.
\newblock {\em Notre Dame Journal of Formal Logic}, 27(1):39--50, 1986.

\bibitem[She75]{shelah1975generalized}
Saharon Shelah.
\newblock Generalized quantifiers and compact logic.
\newblock {\em Transactions of the American Mathematical Society},
  204:342--364, 1975.

\bibitem[She90]{shelah1990classification}
Saharon Shelah.
\newblock {\em Classification theory: and the number of non-isomorphic models}.
\newblock Elsevier, 1990.

\end{thebibliography}

\end{document}